\documentclass[a4paper,12pt]{article}
\usepackage[T1]{fontenc}
\usepackage{amsmath, graphicx, color, amsthm, amssymb,tikz}
\usepackage[font=small,labelfont=bf,width=12.5cm]{caption}
\usepackage{algorithmic, algorithm}
\usepackage{MnSymbol}

\def\lf{\left\lfloor}
\def\rf{\right\rfloor}

\newtheorem{theorem}{Theorem}[section]

\newtheorem{proposition}[theorem]{Proposition}
\newtheorem{conjecture}[theorem]{Conjecture}

\newtheorem{problem}[theorem]{Problem}

\title{Colorings with neighborhood parity condition}

\author
{
	Mirko Petru\v{s}evski \thanks{Department of Mathematics and Informatics, Faculty of Mechanical Engineering - Skopje, Republic of North Macedonia. E-Mail: \texttt{mirko.petrushevski@gmail.com}},
	\quad
	Riste \v{S}krekovski\thanks{FMF, University of Ljubljana \& Faculty of Information Studies, Novo mesto, Slovenia. E-Mail: \texttt{skrekovski@gmail.com}}
}

\begin{document}
\maketitle

\begin{abstract}
 In this short paper, we introduce a new vertex coloring whose motivation comes from our series on odd edge-colorings of graphs. A proper vertex coloring $\varphi$ of a graph $G$ is said to be odd if for each non-isolated vertex $x\in V(G)$ there exists a color $c$ such that $\varphi^{-1}(c)\cap N(x)$ is odd-sized. We prove that every simple planar graph admits an odd $9$-coloring, and conjecture that $5$ colors always suffice.
\end{abstract}

\medskip



\section{Introduction}

 All considered graphs in this paper are simple, finite and undirected. We follow~\cite{BonMur08} for all terminology and notation not defined here. A $k$-(vertex-)coloring of a graph $G$ is an assignment $\varphi: V(G)\to\{1,\ldots,k\}$. A coloring $\varphi$ is said to be \textit{proper} if every color class is an independent subset of the vertex set of $G$. A \textit{hypergraph} $\mathcal{H}=(V(\mathcal{H}),\mathcal{E}(\mathcal{H}))$ is a generalization of a graph, its (hyper-)edges are subsets of $V(\mathcal{H})$ of arbitrary positive size. There are various notions of (vertex-)coloring of hypergraphs, which when restricted to graphs coincide with proper graph coloring. One such notion was introduced by Even at al.~\cite{EveLotRonSmo03} (in a geometric setting) in connection with frequency assignment problems for cellular networks, as follows. A coloring of a hypergraph $\mathcal{H}$ is \textit{conflict-free} \textit{(CF)} if for every edge $e\in \mathcal{E}(\mathcal{H})$ there is a color $c$ that occurs exactly once on the vertices of $e$. The \textit{CF chromatic number} of $\mathcal{H}$ is the minimum $k$ for which $\mathcal{H}$ admits a CF $k$-coloring. For graphs, Cheilaris~\cite{Che09} studied the \textit{CF coloring with respect to neighborhoods}, that is, the coloring in which for every non-isolated vertex $x$ there is a color that occurs exactly once in the (open) neighborhood $N(x)$, and proved the upper bound $2\sqrt{n}$ for the CF chromatic number of a graph of order $n$. For more on CF colorings see, e.g., \cite{CheTot11, GleSzaTar14, KosKumLuc12, PacTar09, Smo13}.

A similar but considerably less studied notion (concerning a weaker requirement for the occurrence of a color) was introduced by Cheilaris et al.~\cite{CheKesPal13} as follows. An \textit{odd coloring} of a hypergraph $\mathcal{H}$ is a coloring such that for every edge $e\in \mathcal{E}(\mathcal{H})$ there is a color $c$ with an odd number of vertices of $e$ colored by $c$. Particular features of the same notion (under the name \textit{weak-parity coloring}) have been considered by Fabrici and G\"{o}ring~\cite{FabGor16} (in regard to face-hypergraphs of planar graphs) and also by Bunde et al.~\cite{BunMilWesWu07} (in regard to coloring of graphs with respect to paths, i.e., path-hypergraphs). For various edge colorings of graphs with parity condition required at the vertices we refer the reader to~\cite{AtaPetSkr16, Pet15, KanKatVar18, LuzPetSkr15, LuzPetSkr18, Pet18, PetSkr21}.

\medskip

 In this paper we study certain aspects of odd colorings for graphs with respect to (open) neighborhoods, that is, the colorings of a graph $G$ such that for every non-isolated vertex $x$ there is a color that occurs an odd number of times in the neighborhood $N_G(x)$. Our focus is on colorings that are at the same time proper, and we mainly confine to planar graphs.

\smallskip

Let us denote by $\chi_o(G)$ the minimum number of colors in any proper coloring of a given graph $G$ that is odd with respect to neighborhoods, call this the \textit{odd chromatic number} of $G$. Note that the obvious inequality $\chi(G)\leq\chi_o(G)$ may be strict; e.g., $\chi(C_4)=2$ whereas $\chi_o(C_4)=4$. Similarly, $\chi(C_5)=3$ whereas $\chi_o(C_5)=5$. In fact, the difference $\chi_o(G)-\chi(G)$ can acquire arbitrarily large values. Indeed, let $G$ be obtained from $K_n$ $(n\geq2)$ by subdividing each edge once. Since $G$ is bipartite, $\chi(G)=2$. On the other hand, it is readily seen that $\chi_o(G)\geq n$. Note in passing another distinction between the chromatic number and the odd chromatic number. The former graph parameter is monotonic in regard to the `subgraph relation', that is, if $H\subseteq G$ then $\chi(H)\leq\chi(G)$. This nice monotonicity feature does not hold for the odd chromatic number in general. For example, $C_4$ is a subgraph of the kite $K_4-e$, but nevertheless we have $\chi_o(C_4)=4>3=\chi_o(K_4-e)$.

\medskip

The Four Color Theorem~\cite{AppHak77,RobSanSeyTho96} asserts the tight upper bound $\chi(G)\leq4$ for the chromatic number of any planar graph $G$. One naturally starts wondering about an analogous bound for the odd chromatic number of all planar graphs. Since $\chi_o(C_5)=5$, four colors no longer suffice. It is our belief that five colors always suffice.

\begin{conjecture}
    \label{conj:1}
For every planar graph $G$ it holds that $\chi_o(G)\leq 5$.
\end{conjecture}

The bound in Conjecture~\ref{conj:1} is achieved for infinitely many planar graphs. Namely, let $\mathcal{F}=\{G: G$ is a non-trivial connected graph such that every block of $G$ is isomorphic to $C_5\}$.

\begin{proposition}
    \label{prop:F}
If $G\in\mathcal{F}$ then $\chi_o(G)=5$.
\end{proposition}

\begin{proof}
Consider a counter-example $G$ with minimum number of blocks, $b(G)$. Since $G\neq C_5$, we have that $b(G)\geq2$. Let $B$ be an endblock of $G$, and let $v_1,v_2,v_3,v_4,v_5$ be the vertices of the $5$-cycle $B$ met on a circular traversing that starts at the unique cut vertex $v_1$ contained within $B$. Consider the graph $G'=G-\{v_2,v_3,v_4,v_5\}$. Since $G'\in\mathcal{F}$ and $b(G')=b(G)-1$, it holds that $\chi_o(G')=5$. Take an odd coloring of $G'$ with color set $\{1,2,3,4,5\}$ such that $v_1$ is colored by $1$ and the color $4$ has odd number of occurrences in $N_{G'}(v_1)$. Extend to $V(G)$ by assigning the color $i$ to $v_i$, for $i=2,3,4,5$. This clearly gives an odd $5$-coloring of $G$. Consequently, since we are supposing $\chi_o(G)\neq5$, there must exist an odd coloring $c$ of $G$ with color set $\{1,2,3,4\}$. Upon permutation of colors, we may assume that $c(v_i)=i$ for $1\le i\le4$. Hence, it must be that $c(v_5)=2$. Thus, any color which satisfies the odd-condition in $N_G(v_1)$ has an odd number of occurrences in $N_{G'}(v_1)$ as well. It follows that the restriction of $c$ to $G'$ is an odd $4$-coloring, contradicting the minimality choice of $G$.
\end{proof}

 The main purpose of this paper is to provide the first support to Conjecture~\ref{conj:1} by proving the following.

\begin{theorem}
    \label{mainresult}
For every planar graph $G$ it holds that $\chi_o(G)\leq 9$.
\end{theorem}

The rest of the article is organized as follows. In the next section we provide a proof of Theorem~\ref{mainresult}. And in Section~3 we briefly convey some of our ideas for possible further work on the topic of proper colorings of graphs that are odd with respect to neighborhoods.

\section{Proof of Theorem~1.3}

Throughout, we will refer to a coloring of this kind as a \textit{nice coloring}; that is, a nice coloring is a (proper) odd coloring which uses at most $9$ colors. Arguing by contradiction, let $G$ be a counter-example of minimum order $n=n(G)$. Clearly, $G$ is connected and has $n\geq10$ vertices. We proceed to exhibit several structural constraints of $G$.

\bigskip

\noindent \textbf{Claim 1.} \textit{The minimum degree $\delta(G)$ equals $5$.}

\medskip

\noindent Since $G$ is a connected planar simple graph with $n>2$ vertices, the inequality $\delta(G)\leq5$ is a consequence of Euler's formula. Consider a vertex $v$ of degree $d_G(v)=\delta(G)$. If $d_G(v)=1$ or $3$, then take a nice coloring of $G-v$. By forbidding at most six colors at $v$ (namely, at most three colors used for $v$'s neighbors and at most three additional colors in regard to oddness concerning the neighborhoods in $G-v$ of these vertices), the coloring extends to a nice coloring of $G$.

Suppose next that $d_G(v)=2$, and say $N_G(v)=\{x,y\}$. Construct $G'$ by removing $v$ from $G$ plus connecting $x$ and $y$ if they are not already adjacent. By minimality, $G'$ admits a nice coloring $c$.  Say $c(x)=1$ and $c(y)=2$, and let color(s) $1'$ and $2'$, respectively, have odd number of occurrences in $N_{G'}(x)$ and $N_{G'}(y)$. If there are several possibilities for $1'$, then choose $1'\neq2$; do similarly for $2'$ in regard to $1$. Extend the coloring to $G$ by using for $v$ a color different from $1,2,1',2'$. The properness of the coloring is clearly preserved. As for the weak-oddness concerning neighborhoods, $v$ is fine because $1\neq2$. If $1'\neq2$ then $x$ is fine since $1'$ remains to be odd on $N_G(x)$. Contrarily, if $1'=2$ then $c(v)$ is odd on $N_G(x)$. Similarly, the vertex $y$ is also fine.

Finally, suppose $d_G(v)=4$ and let $w\in N_G(v)$. Remove $v$ and connect $w$ by an edge to every other non-adjacent neighbor of $v$. The obtained graph $G''$ is simple and planar. Indeed, it can be equivalently obtained from $G$ by contracting the edge $vw$ (if parallel edges arise through possible mutual neighbors of $v$ and $w$, then for each such adjacency a single edge is kept and its copies are deleted). Notice that under any nice coloring of $G''$, the color of $w$ occurs exactly once in $N_G(v)$. Therefore, any such coloring extends to a nice coloring of $G$ by forbidding at most eight colors at $v$.

We conclude that $d_G(v)=5$.\qed

\bigskip

We refer to any vertex of degree $d$ as a \textit{$d$-vertex}. Similarly, a vertex of degree at least $d$ is a \textit{$d^+$-vertex}. Analogous terminology applies to faces in regard to a planar embedding of $G$.

\bigskip

\noindent \textbf{Claim 2.} \textit{If $v$ is a $5$-vertex of $G$, then it has at most one neighbor of odd degree.}

\medskip

\noindent Suppose each of two vertices $u,w\in N_G(v)$ has an odd degree in $G$. Consider a nice coloring $c$ of $G-v$. We intend to extend $c$ to $G$. Since $N_G(u)$ and $N_G(w)$ are odd-sized, no color is blocked at $v$ in regard to oddness in the neighborhoods of $u$ and $w$. Moreover, as $N_G(v)$ is odd-sized as well, the oddness of a color in this particular neighborhood is guaranteed. Therefore, by forbidding at most eight colors at $v$ (all of $c(N_G(v))$ and at most $3$ additional colors concerning oddness in neighborhoods of the vertices forming $N_G(v)\backslash\{u,w\}$), the coloring $c$ extends to a nice coloring of $G$.\qed

\bigskip

Since the graph $G$ is connected, for an arbitrary planar embedding Euler's formula gives
$$|V(G)|-|E(G)|+|F(G)|=2\,,$$ where $F(G)$ is the set of faces.
Our proof is based on the discharging technique. We assign initial charges to the vertices and faces according to the left-hand side of the following equality (which immediately follows from Euler's formula):
$$\sum_{v\in V(G)}(d(v)-6)+\sum_{f\in F(G)}(2d(f)-6)=-12\,.$$

Thus, any vertex $v$ receives charge $d(v)-6$, and any face $f$ obtains charge $2d(f)-6$. As $G$ is simple, for every face $f$ it holds that $d(f)\geq3$, implying that its initial charge is non-negative. By Claim~1, the only vertices that have negative initial charge are the $5$-vertices of $G$ (each is assigned with charge $-1$).

\bigskip

\noindent \textbf{Claim 3.} \textit{Any $d$-face $f\in F(G)$ is incident with at most $\lf \frac{2d}{3} \rf$ $5$-vertices.}

\medskip

\noindent By Claim~2, on a facial walk of $f$ no three consecutive vertices are of degree $5$. From this, the stated upper bound for the number of $5$-vertices incident with $f$ follows immediately.\qed

\bigskip

We use the following discharging rules (cf. Figure~\ref{fig:rules}):

\begin{itemize}

\item[(R1)] Every $4^+$-face $f$ sends charge $1$ to any incident $5$-vertex $v$.

\item[(R2)] Every $8^+$-vertex $u$ sends charge to any adjacent $5$-vertex $v$ if both faces incident with the edge $uv$ are triangles, as follows: if one of these triangles has another $5$-vertex (besides $v$), then the sent charge equals $\frac{1}{3}$; otherwise, the sent charge equals $\frac{1}{2}$. (In view of Claim~2, it cannot be that both these triangles have a $5$-vertex $\neq v$.)

\item[(R3)] Every $4^+$-face $f$ sends through every incident edge $uw$ with $d(u)=d(w)=6$ charge $\frac{1}{2}$ to a $5$-vertex $v$ if $uvw$ is a triangular face and $f$ is not a $4$-face incident with two $5$-vertices.

\end{itemize}

\begin{figure}[ht!]
	$$
		\includegraphics[scale=0.8]{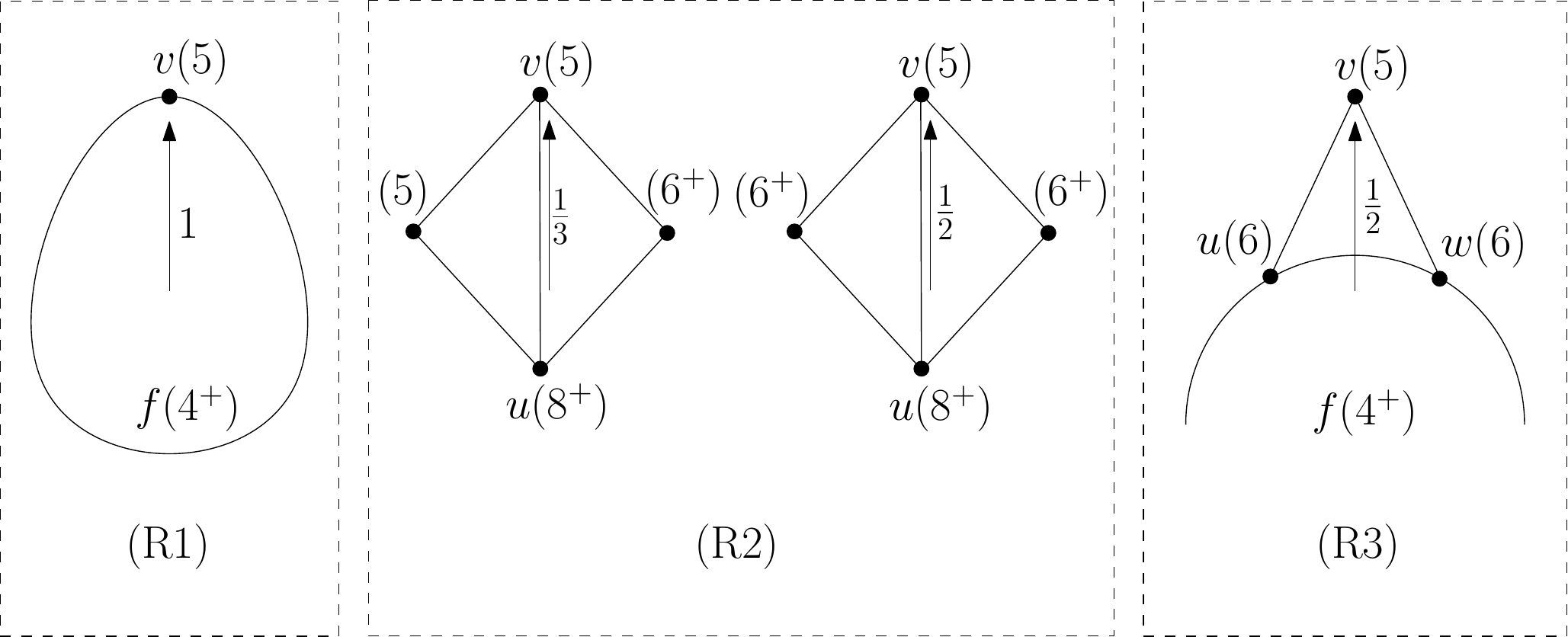}
	$$
	\caption{In (R2), the numbers standing in brackets beside vertices or faces indicate their degrees. In (R3), $f$ is not a $4$-face incident with two $5$-vertices and $uvw$ is an adjacent triangular face.}
	\label{fig:rules}
\end{figure}

\bigskip

\noindent \textbf{Claim 4.} \textit{No $8^+$-vertex becomes negatively charged by applying $(R2)$.}

\medskip

\noindent Suppose there is a $d$-vertex $u$ with $d\geq8$ that becomes negatively charged by applying Rule~2. Recall that its initial charge was $d-6$. Consider a circular ordering of $N_G(u)$ in regard to the embedding of $G$. Note that no three consecutive neighbors $v_i,v_{i+1},v_{i+2}\in N_G(u)$ have received charge from $u$ during the discharging process, for otherwise $v_{i+1}$ contradicts Claim~2. Therefore, $u$ gave charge to at most $\lf \frac{2d}{3} \rf$ of its neighbors, and at most $\frac{1}{2}$ of charge per neighbor. Consequently, $d-6<\frac{2d}{3}\cdot\frac{1}{2}$. Equivalently, $d<9$. So $u$ is an $8$-vertex. Moreover, since we are supposing that it is negatively charged after applying Rule~2, the vertex $v$ gave charge to precisely five $5$-vertices. However, then it is easily seen that the initial charge $2$ of $v$ is reduced by at most $\frac{11}{6}=\frac{1}{3}+\frac{1}{3}+\frac{1}{3}+\frac{1}{3}+\frac{1}{2}$, a contradiction.\qed

\bigskip

\noindent \textbf{Claim 5.} \textit{No $4^+$-face becomes negatively charged by applying $(R1)$ or $(R3)$.}

\medskip

\noindent Suppose there is a $d$-face $f$ that becomes negatively charged by applying Rules~1 or~3. Recall that its initial charge was $2d-6$. For the purposes of this proof, it is useful to think of the charge sent by $f$ according to Rule~3 as follows. Consider any maximal sequence of `consecutive' triangular faces that are adjacent to $f$, and split the sequence into $2$-tuples and a possible $1$-tuple. For every such $2$-tuple of triangular faces, the total charge $1=\frac{1}{2}+\frac{1}{2}$ first goes to their common $6$-vertex, and only afterwards splits evenly and reaches the targeted $5$-vertices (possibly the same). And for every such $1$-tuple the charge $\frac{1}{2}$ coming from $f$ to this face first splits evenly and goes to the pair of $6$-vertices shared with $f$, and only afterwards reaches the targeted $5$-vertex.

 With this perspective in mind, consider a facial walk of $f$ and notice that during the discharging it gives away to each incident vertex charge of either $0,\frac{1}{4}$, or $1$; moreover, no three consecutive vertices receive from $f$ charge $1$ (according to Claim~2). Hence, on at most $\lf \frac{2d}{3} \rf$ occasions $f$ gives charge $1$, and on the other $\lf \frac{d+1}{3} \rf$ occasions $f$ gives charge at most $\frac{1}{4}$. Consequently, $2d-6<\frac{2d}{3}+\frac{d+1}{12}$. Equivalently, $d<\frac{73}{15}$. So $f$ is a $4$-face. However, it is easily seen that a $4$-face gives away at most $2$ of its initial charge (which also equals $2$), a contradiction. \qed

\bigskip

Since the total charge remains negative (it equals $-12$), from Claims~4 and 5 it follows that there is a $5$-vertex $v$ which remains negatively charged even after applying the discharging rules. By (R1) and (R2), we have the following.

\bigskip

\noindent \textbf{Claim 6.} \textit{The $5$-vertex $v$ has only $3$-faces around it, and it is a neighbor of at most two $8^+$-vertices. Moreover, if $v$ neighbors exactly two $8^+$-vertices, then these three vertices have another $5$-vertex as a common neighbor.}

%

\bigskip

Let $v_1,v_2,v_3,v_4,v_5$ be the neighbors of $v$ in a circular order regarding the considered plane embedding, i.e., such that $vv_1v_2,vv_2v_3,vv_3v_4,vv_4v_5,vv_5v_1$ are the five $3$-faces incident with $v$.

\bigskip

\noindent \textbf{Claim 7.} \textit{For some $j=1,2,\ldots,5$, the vertices $v_{j},v_{j+2}$ are not adjacent and their only common neighbors are $v,v_{j+1}$.}

\medskip

\noindent Arguing by contradiction, suppose that every pair of vertices $v_{j},v_{j+2}$ are either adjacent or have a common neighbor $\neq v,v_{j+1}$. This readily implies the existence of a vertex $x\neq v$ that is adjacent to all five vertices $v_1,v_2,v_3,v_4,v_5$.
Now take a nice coloring $c$ of $G-v$, and let $s=|c(N_G(v))|$. In view of the $5$-cycle $v_1v_2v_3v_4v_5$, it holds that $s\in\{3,4,5\}$. In case $s=3$, the coloring $c$ extends to a nice coloring of $G$ since at most $8$ colors are forbidden at $v$. These are the three colors used on $N_G(v)$ and at most five additional colors in regard to `oddness' in $N_{G-v}(v_i)$ for $i=1,2,\ldots,5$. Hence $s\in\{4,5\}$, and from this we are able to further deduce that there are precisely $s$ vertices $v_i (i=1,2,\ldots,5)$ such that $c(v_{i-1})\neq c(v_{i+1})$.

Unless the coloring $c$ extends to $G$, each of the nine available colors is blocked at the vertex $v$, either due to properness or to `oddness'. Consequently, there are at least $(9-s)$ vertices among the $v_i$'s each of which blocks at $v$ a separate color $\neq c(v_1),c(v_2),c(v_3),c(v_4),c(v_5)$ in regard to `oddness' in its neighborhood within $G-v$. Clearly, any such $v_i$ is of even degree in $G$ and any color from $\{c(v_{i-1}),c(v_{i+1})\}$ appears an even number of times in $N_{G-v}(v_i)$. Moreover, for at least four such vertices $v_i$, say $v_1,v_2,v_3,v_4$, it also holds that $c(v_{i-1})\neq c(v_{i+1})$. Noting that the color $c(x)$ occurs in every of the four neighborhoods $N_{G-v}(v_i)$ for $i=1,\ldots,4$, it follows that $c(x)$ has an odd number of appearances in such a neighborhood for at most one vertex $v_i$. Therefore, at least three $v_i$'s are of degree $\geq 2+2+2+1=7$ in $G-v$. (Namely, each of the colors $c(v_{i-1}),c(v_{i+1}),c(x)$ yields a summand of $2$, and the summand of $1$ refers to the color blocked by $v_i$ due to `oddness' in $N_{G-v}(v_i)$).
However, this yields at least three $8^+$-neighbors of $v$ in $G$, contradicting Claim~6.\qed

\bigskip

Let $v_j,v_{j+2}$ fulfil Claim~7. Consider the graph $(G-v)/\{v_j,v_{j+2}\}$ obtained from $G-v$ by identifying $v_j$ and $v_{j+2}$, and then deleting one of the arising two parallel edges (between the new vertex $\{v_j,v_{j+2}\}$ and the vertex $v_{j+1}$). Take a nice coloring of $(G-v)/\{v_j,v_{j+2}\}$, and look into the inherited coloring $c$ of $G-v$.

\bigskip

\noindent \textbf{Claim 8.} \textit{All $v_i$'s are of even degree in $G$, and the coloring $c$ is nice.}

\medskip

\noindent
Notice that besides the odd-condition at $v_{j+1}$, the coloring $c$ is nice on $G-v$. If $d_G(v_{j+1})$ is odd, then when adding $v$ to $G-v$ we do not need to care about odd-condition at $v_{j+1}$, as it will automatically hold. But then we need to avoid only $8$ colors when coloring $v$, which is doable, a contradiction. Hence $d_G(v_{j+1})$ is even. Consequently, $c$ is a nice coloring of $G-v$ (as $d_{G-v}(v_{j+1})$ is odd and the rest of the conditions are fine). Since $c$ does not extend to a nice coloring of $G$, by reasoning in the same manner as at the beginning of the proof of Claim~7, we are able to conclude that $|c(N_G(v))|\geq4$. On the other hand, by the construction of $c$ we clearly have $|c(N_G(v))|<5$. So it must be that exactly four colors appear in $N_G(v)$. Moreover, as $c$ is non-extendable, each of the vertices $v_1,v_2,v_3,v_4,v_5$ blocks at $v$ a separate color $\notin c(N_G(v))$ in regard to already (and uniquely) fulfilled `oddness' in its neighborhood in $G-v$. However, this readily implies that each $v_i$ is of even degree in $G$.\qed

\begin{figure}[ht!]
	$$
		\includegraphics[scale=0.6]{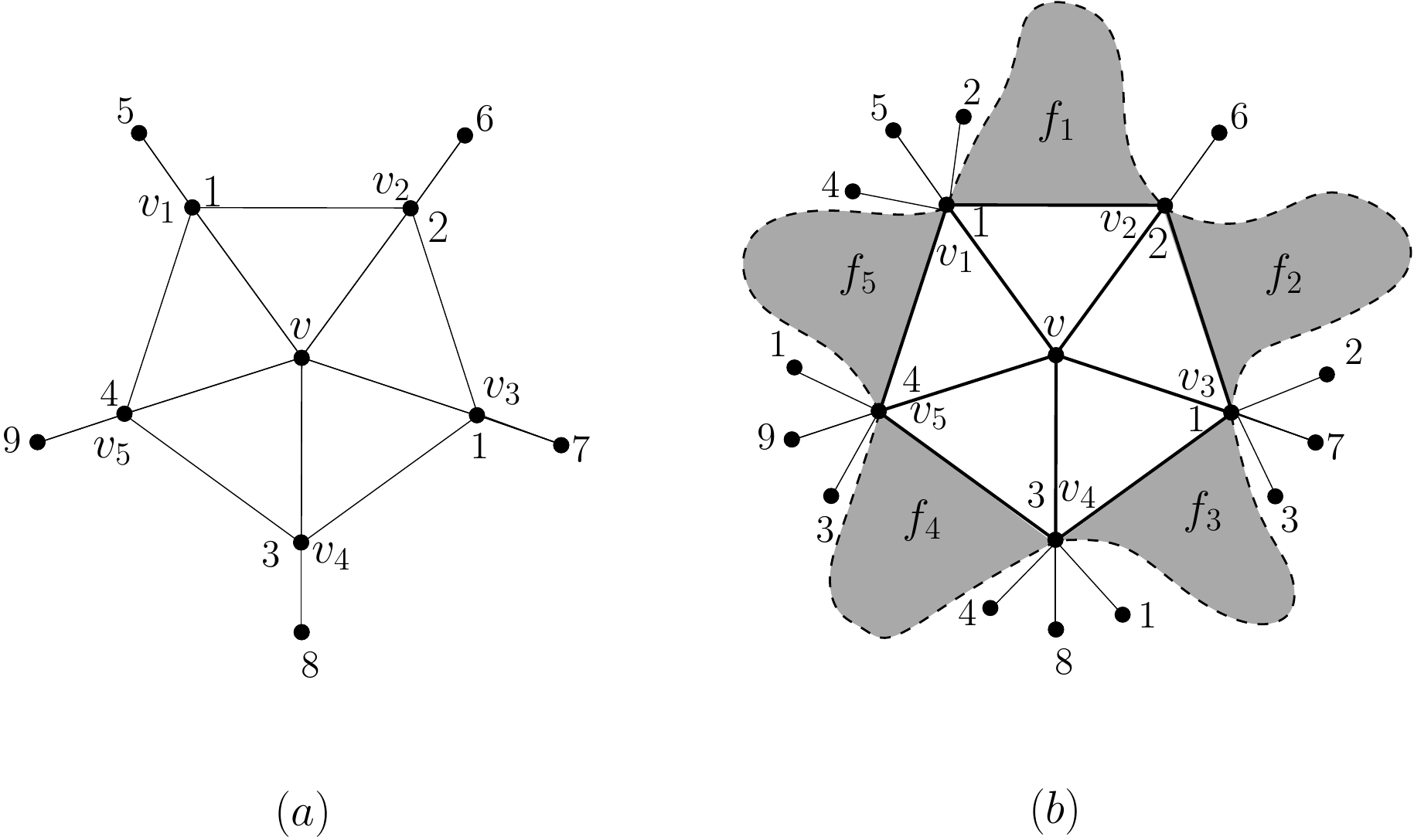}
	$$
	\caption{Local color distribution, guaranteed `free hanging' colors and the faces adjacent with the local triangulation around $v$.}
	\label{fig:localfaces}
\end{figure}

The proof of Claim~8 certifies that, upon permuting colors, the following color distribution occurs under $c$: the colors $1,2,1,3,4$ are used on $v_1,v_2,v_3,v_4,v_5$, respectively, and each of the colors $5,6,7,8,9$ happens to be the only color with an odd number of appearances on a separate neighborhood $N_{G-v}(v_i)$ (cf. Figure~\ref{fig:localfaces}(a)). Moreover, as colors $5,6,7,8,9$ are `exclusive' in regard to `oddness' on $N_{G-v}(v_1),N_{G-v}(v_2),N_{G-v}(v_3),N_{G-v}(v_4),N_{G-v}(v_5)$, respectively, any of the colors $1,2,3,4$ has an even (possibly $0$) number of occurrences  in each set $N_{G-v}(v_i)$ $(i=1,\ldots,5)$. Consequently, the situation depicted in Figure~\ref{fig:localfaces}(b) is present, that is, at four of the $v_i$'s (namely, $v_1,v_2,v_4,v_5$) we are guaranteed two more `free hanging' colors per vertex.
Let $f_i$ $(i=1,\ldots,5)$ be the face incident with $v_iv_{i+1}$ but not with $v$ (cf. Figure~\ref{fig:localfaces}(b)). Note that $f_1,f_2,f_3,f_4,f_5$ need not be pairwise distinct. We say that $f_i$ is \textit{convenient} if it is a $4^+$-face and $v_i,v_{i+1}$ are $6$-vertices whose only common neighbor is $v$.


\bigskip

\noindent \textbf{Claim 9.} \textit{If all of $v_1,v_2,v_3,v_4,v_5$ are $6$-vertices in $G$, then at least two members of the list $f_1,f_2,f_3,f_4,f_5$ are convenient.}

\medskip

\noindent It suffices to observe the following: if a given $f_k$ is a $3$-face, then the vertices $v_k$ and $v_{k+1}$ must share a free hanging color. Since all the $v_i$'s are $6$-vertices, it follows that $v_3$ and $v_4$ do not have a common free hanging color. Consequently, $f_3$ is convenient. Quite similarly, $f_5$ is convenient.\qed

\bigskip

Following the same line of reasoning, we also conclude the following:

\bigskip

\noindent \textbf{Claim 10.} \textit{If precisely one of $v_1,v_2,v_3,v_4,v_5$ is an $8^+$-vertex in $G$, then some $f_i$ is convenient.}

\medskip

\noindent Indeed, we may assume that neither $v_1$ nor $v_5$ is the $8^+$-neighbor of $v$. Then $v_1,v_5$ are $6$-vertices without a common free hanging color. Consequently, $f_5$ is convenient.\qed

\bigskip

Our next (and final) claim assures that every convenient face $f_i$ sends charge $\frac{1}{2}$ to $v$, in accordance with (R3).

\bigskip

\noindent \textbf{Claim 11.} \textit{No convenient face is a quadrangle $ABCD$ where $A,B$ are $5$-vertices and $C,D$ are $6$-vertices.}

\medskip

\noindent Arguing by contradiction, we suppose that the situation depicted in Figure~\ref{fig:nice}(a) is present. Without loss of generality, we may assume that $B$ and $D$ are non-adjacent and $N_G(B)\cap N_G(D)=\{A,C\}$. Indeed, if each of the pairs of vertices $A,C$ and $B,D$ are either adjacent or share a common neighbor outside the set $\{A,B,C,D\}$, then by planarity, the situation depicted in Figure~\ref{fig:nice}(b) occurs. However, then $C$ and $D$ have a common neighbor $\neq v$, which contradicts that the considered face is convenient.

\begin{figure}[ht!]
	$$
		\includegraphics[scale=0.6]{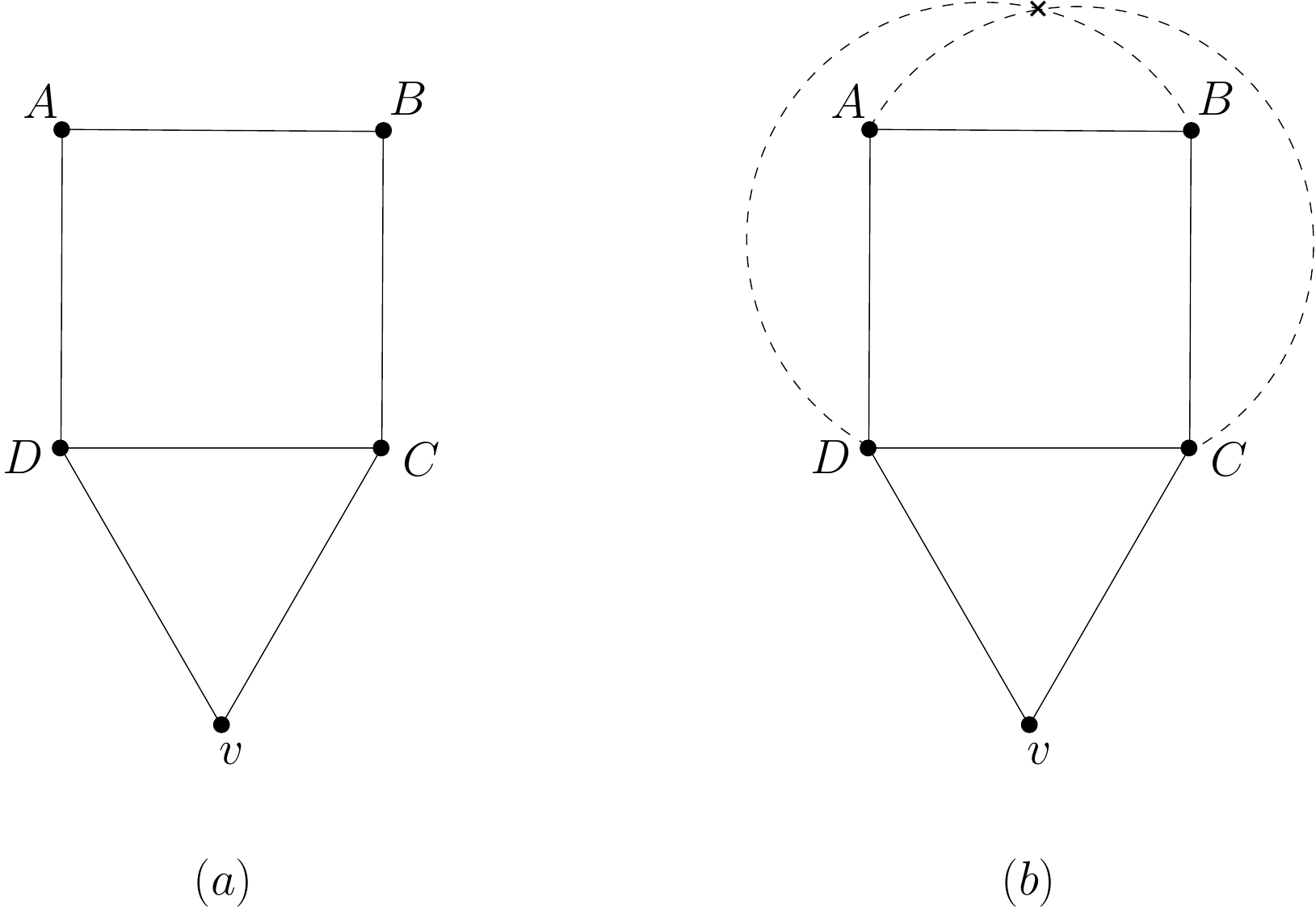}
	$$
	\caption{A convenient face next to the $5$-vertex $v$ that happens to be a quadrangle $ABCD$, where $A,B$ are $5$-vertices and $C,D$ are $6$-vertices.}
	\label{fig:nice}
\end{figure}

With our assumption for the pair $B,D$, we look at the graph $G'=(G-\{v,A,C\})/\{B,D\}$, that is, $G'$ is obtained from $G-\{v,A,C\}$ by identifying $B$ and $D$. Since $G'$ is planar and smaller than $G$, it admits a nice coloring. Consider the inherited partial coloring $c$ of $G$. Note that vertices $B$ and $D$ are colored the same under $c$, and the only uncolored vertices are $v,A$ and $C$. We extend $c$ to a nice coloring of $G$ as follows.

First we color $v$, and for the time being we don't care about preserving oddness in $N_{G-A-C}(D)$. Apart from the color $c(D)$ and possibly a second color in regard to unique oddness in $N_{G-\{v,A\}}(C)$, at most $6=3\cdot 2$ more colors are forbidden at $v$
(by its three neighbors $\neq C,D$ and oddness in their respective neighborhoods within $G-\{v,A,C\}$). Thus there is a color which is available for $v$.

Next we color $C$. There are at most $8=1+1+3\cdot2$ colors forbidden at $C$ (namely, the color $c(B)=c(D)$, the color $c(v)$ and possibly six more colors concerning the other three neighbors of $C$ and oddness in their respective neighborhoods within $G-A$). So there is a color available for $C$. Once $C$ is assigned with an available color, note that there is a color with an odd number of appearances in $N_{G-A}(D)$ (because $d_{G-A}(D)=5$).

Finally, we color $A$. Since this is a $5$-vertex, it suffices to choose a color for $A$ so that the properness of the coloring and the oddness in $N_{G-A}(D)$ are preserved. As $B$ and $D$ are colored the same, and $B$ is a $5$-vertex, the number of forbidden colors at $A$ is at most $8=1+1+3\cdot2$ (the color $c(B)=c(D)$, possibly a second color in regard to unique oddness in $N_{G-A}(D)$, and six more colors concerning the other three neighbors of $A$ and oddness in their respective neighborhoods within $G-A$). We conclude that $c$ indeed extends to a nice coloring of $G$, a contradiction.\qed

\bigskip

From Claims~9-11 it follows that the vertex $v$ receives charge at least $1$ during the discharging process, hence it cannot remain to be negatively charged. This contradiction concludes our proof.

\section{Further work}
By Proposition~\ref{prop:F}, there exist plenty of planar graphs with girth $5$ that require five colors for an odd coloring. We are tempted to propose the following.

\begin{conjecture}
\label{odd4}
Every planar graph $G$ of girth at least $6$ has $\chi_o(G)\leq4$.
\end{conjecture}

If true, Conjecture~\ref{odd4} would imply the Four Color Theorem (4CT). Indeed, say $G$ is an arbitrary planar graph. Subdivide every edge of $G$ once, i.e., consider the complete subdivision $S(G)$. As girth $g(S(G))\ge6$, take an odd $4$-coloring $c$ of $S(G)$. The restriction of $c$ to $V(G)$ is a proper $4$-coloring of $G$.

\medskip

Although the problem of odd coloring is intriguing already in the class of plane graphs,
the concept can be naturally studied also for graphs embedded in higher surfaces.
For a surface $\Sigma$, we define the \textit{odd chromatic number} of $\Sigma$,
\[\chi_o(\Sigma) = \max_{G\hookrightarrow \Sigma}\chi_o(G),\]
as the maximum of $\chi_o(G)$ over all graphs $G$ embedded into $\Sigma$. Conjecture~\ref{conj:1} and Proposition~\ref{prop:F} combined purport that $\chi_o(S_0) = 5$, where $S_0$ is the sphere. It would be interesting
to have a similar characterization to the Heawood number for other surfaces of higher
genus.
\begin{problem}
Determine $\chi_o(\Sigma)$ for surfaces $\Sigma$ of higher genus.
\end{problem}

Possibly the odd chromatic number of higher surfaces coincides with the corresponding Heawood number.

\bigskip
\noindent
{\bf Acknowledgements.}
This work is partially supported by ARRS Program P1-0383 and ARRS Project J1-3002.

\bibliographystyle{plain}

\begin{thebibliography}{99}
\bibitem{AppHak77} K.~Appel and W.~Haken, \textit{The solution of the Four-Color Map Problem}, Sci. Amer. $\mathbf{237}$, (1977) 108--121.
\bibitem{AtaPetSkr16} R.~Atanasov, M.~Petru\v{s}evski, R.~\v{S}krekovski, \textit{Odd edge-colorability of subcubic graphs}, Ars Math. Contemp. $\mathbf{10}$ (2016), 359--370.
\bibitem{BonMur08} J.~A.~Bondy and U.~S.~R.~Murty, \textit{Graph Theory}, Graduate Texts in Mathematics, Springer, New York $\mathbf{244}$ (2008).
\bibitem{BunMilWesWu07} D.~P.~Bunde, K.~Milans, D.~B.~West and H.~Wu, \textit{Parity and strong parity edge-coloring of graphs}, Congr. Numer. $\mathbf{187}$ (2007) 193--213.
\bibitem{Che09} P.~Cheilaris, \textit{Conflict-Free Coloring}, PhD Thesis (City University of New York, 2009).
\bibitem{Pet15} M.~Petru\v{s}evski, \textit{A note on weak odd edge-colorings of graphs}, Adv. Math. Sci. Journal $\mathbf{4}(1)$ (2015) 7--10.
\bibitem{CheKesPal13} P.~Cheilaris, B.~Keszegh and D.~P\'{a}lv\"{o}lgyi, \textit{Unique-maximum and conflict-free colouring for hypergraphs and tree graphs}, SIAM J. Discrete Math $\mathbf{27}$ (2013) 1775--1787.
\bibitem{CheTot11} P.~Cheilaris and G.~T\'{o}th, \textit{Graph unique-maximum and conflict-free coloring}, J. Discrete Algorithms $\mathbf{9}$ (2011) 241--251.
\bibitem{EveLotRonSmo03} G.~Even, Z.~Lotker, D.~Ron and S.~Smorodinsky, \textit{Conflict-free colorings of simple geometric regions with applications to frequency assignment in cellular networks}, SIAM J. Comput. $\mathbf{33}$ (2003) 94--136.
\bibitem{FabGor16} I.~Fabrici and F.~G\"{o}ring, \textit{Unique-maximum coloring of plane graphs}, Dicussiones Mathematicae Graph Theory $\mathbf{36}$ (2016) 95--102.
\bibitem{GleSzaTar14} R.~Glebov, T.~Szab\'{o} and G.~Tardos, \textit{Conflict-free colorings of graphs}, Combin. Probab. Comput. $\mathbf{23(3)}$ (2014) 434--448.
\bibitem{KanKatVar18} M.~Kano, G.~Y.~Katona, K.~Varga, \textit{Decomposition of a graph into two disjoint odd subgraphs}, Graphs and Combin. $\mathbf{34}(6)$ (2018) 1581--1588.
\bibitem{KosKumLuc12} A.~V.~Kostochka, M.~Kumbhat and T.~{\L}uczak, \textit{Conflict-free colorings of uniform hypergraphs with few edges}, Combin. Probab. Comput. $\mathbf{21}$ (2012) 611--622.
\bibitem{LuzPetSkr15} B.~Lu\v{z}ar, M.~Petru\v{s}evski, R.~\v{S}krekovski, \textit{Odd edge coloring of graphs}, Ars Math. Contemp. $\mathbf{9}$ (2015)  277--287.
\bibitem{LuzPetSkr18} B.~Lu\v{z}ar, M.~Petru\v{s}evski, R.~\v{S}krekovski, \textit{On vertex-parity edge-colorings}, J. Combin. Optim. $\mathbf{35}(2)$ (2018) 373--388.
\bibitem{PacTar09} J.~Pach and G.~Tardos, \textit{Conflict-free colorings of graphs and hypergraphs}, Combin. Probab. Comput. $\mathbf{18}$ (2009) 819--834.
\bibitem{Pet18} M.~Petru\v{s}evski, \textit{Odd $4$-edge-colorability of graphs}, J. Graph Theory $\mathbf{87}(4)$ (2018) 460--474.
\bibitem{PetSkr21} M.~Petru\v{s}evski, R.~\v{S}krekovski, \textit{Odd decompositions and coverings of graphs}, European J. Combin. $\mathbf{91}$ (2021) 103225.
\bibitem{RobSanSeyTho96} N.~Robertson, D.~P.~Sanders, P.~D.~Seymour and R.~Thomas, \textit{A New Proof of the Four Colour Theorem}, Electron. Res. Announc. Amer. Math. Soc. \textbf{2} (1996) 17--25.
\bibitem{Smo13} S.~Smorodinsky, \textit{Conflict-free coloring and its applications}, In: I.~B\'{a}r\'{a}ny, K.~J.~B\"{o}r\"{o}czky, G.~F.~,~T\'{o}th and J.~Pach (eds) Geometry-Intuitive, Discrete and Convex. Bolyai Society Mathematical Studies, $\mathbf{24}$ Springer, Berlin, Heidelberg (2013).








\end{thebibliography}

\end{document}